\documentclass[10pt,a4paper]{amsart}
\usepackage{color}
\usepackage[all,cmtip,ps]{xy}
\xyoption{rotate}
\usepackage{amsmath}
\usepackage{amsthm}
\usepackage{amssymb}
\usepackage{amscd} 
\usepackage{dashrule}
\usepackage{amsfonts}
\newcommand{\bigslant}[2]{{\raisebox{.1em}{$#1$}\left/\raisebox{-.1em}{$#2$}\right.}}

\input
\theoremstyle{plain}\newtheorem{theorem}{Theorem}[section]
\theoremstyle{plain}
\theoremstyle{plain}\newtheorem{lemma}[theorem]{Lemma}
\theoremstyle{plain}\newtheorem{proposition}[theorem]{Proposition}

\theoremstyle{plain}\newtheorem{definition}[theorem]{Definition}
\theoremstyle{definition}\newtheorem{example}[theorem]{Example}
\theoremstyle{definition}\newtheorem{question}[theorem]{Question}
\theoremstyle{definition}\newtheorem{remark}[theorem]{Remark}
\theoremstyle{definition}

\def\Aff{\mathbf {Aff}_k}
\def\Spec{\mathrm {Spec\; }}

\def\Hom{{\mathrm {Hom}}}
\def\Cov{\mathrm {Cov}}

\def\cO{{\mathcal O}}
\def\cC{{\mathcal C}}

\def\cT{{\mathcal T}}

\def\Im{{\rm Im}}

\def\lMod{{\rm\hbox{-}Mod}}

\def\Coker{{\mathrm {Coker}}} 
\def\Ker{{\mathrm {Ker}}}

\def\id{{\rm id}}

\author{Yves Andr\'e}
\address[Y. Andr\'e]{Sorbonne-Universit\'e , Institut de Math\'ematique de jussieu, 4 Place
Jussieu, 75005 Paris, France.}
\email{yves.andre@imj-prg.fr}

\author{Luisa Fiorot}
\address[L. Fiorot]{Dipartimento di Matematica ``Tullio Levi-Civita'', Universit\`a degli Studi di Padova\\
Via Trieste, 63\\
35121 Padova, Italy.}
\email{luisa.fiorot@unipd.it}

\begin{document}
\title[On the canonical, fpqc, and finite topologies]{On the canonical, fpqc, and finite topologies on affine schemes.
 The state of the art.}

\begin{sloppypar}
{\small\begin{abstract} This is a systematic study of the behaviour of finite coverings of (affine) schemes with regard to two Grothendieck topologies: the canonical topology and the fpqc topology. The history of the problem takes roots in the foundations of Grothendieck topologies, passes through main strides in Commutative Algebra and leads to new Mathematics up to perfectoids and prisms. 

 We first review the canonical topology of affine schemes and show, keeping with Olivier's lost work, that it coincides with the effective descent topology. Covering maps are given by universally injective ring maps, which we discuss in detail.  

We then give a ``catalogue raisonn\'e" of examples of finite coverings which separate the canonical, fpqc and fppf topologies. The key result is that {\it finite coverings of regular schemes are coverings for the canonical topology, and even for the fpqc topology} (but not necessarily for the fppf topology). We discuss a ``{weakly} functorial" aspect of this result.

``Splinters" are those affine Noetherian schemes for which every finite covering is a covering for the canonical topology. We present in geometric terms the state of the art about them.   
We also investigate their mysterious fpqc analogs and prove that, in prime characteristic, they are all regular. This leads us to the problem of descent of regularity by (non necessarily flat) morphisms which are coverings for the fpqc topology, which is settled thanks to a recent theorem of Bhatt-Iyengar-Ma. 
  \end{abstract}}
   \maketitle
{\footnotesize{ \tableofcontents}}

 \section*{Introduction}

 \subsection{}  In Algebraic Geometry, {\it finite coverings}, i.e. finite surjective morphisms, play a fundamental role. For instance, Noether's normalization theorem allows us to identify ``up to finite covering" any sequence of nested subvarieties of an affine algebraic variety to a flag of finite-dimensional vector spaces.

 The present paper is an exposition of the behaviour of finite coverings of affine schemes with regard to two 
 distinguished Grothendieck topologies: the {\it canonical topology} and the {\it fpqc topology}. 
This will lead us from the origins of Grothendieck topologies (FGA, SGA 4) to a contemporary area of Commutative Algebra which is far from being a mere Chapter 0 of Algebraic Geometry (as labeled in EGA),   and beyond, up to perfectoids and prisms.

 \subsection{} The fpqc\footnote{fid\`element plate quasi-compacte.} topology is well-known to every contemporary algebraic geometer in view of its eminent role in descent problems: descent of objects and morphisms, and descent of their properties. 
 Much less known is the canonical topology which, by definition, is the finest topology for which all representable presheaves are sheaves. Grothendieck proved that the fpqc topology is coarser than the canonical topology. 

Strictly coarser? The answer - {\it yes} - is not so well-known but was already given by Raynaud and Gruson, who credit Ferrand for the question \cite[II.1.4]{RG}\footnote{a common misconception is that the existence of non flat split morphism settles the issue: beware that one is dealing with the fpqc topology, not the fpqc pretopology, and that split morphisms are coverings  for every Grothendieck topology.}. In fact, there exist finite coverings which are coverings for the canonical topology but not for the fpqc topology. We will discuss several examples (notably one inspired by the theory of Frobenius algebras \ref{frobalg}).

\subsection{} In part I, we review the canonical topology. In SGA 4, exp. 2, one reads that coverings for the canonical topology are ``universal strict epimorphisms"; this is only stated as a definition, but they are indeed universal strict epimorphisms\footnote{as is explained in \cite[IV. 4.3.9]{SGA3}}! On the other hand, in the context of affine schemes, it turns out that ``strict" is superfluous: coverings for the canonical topology correspond dually to {\it pure} homomorphisms, i.e. to universally injective {ring maps}. 

 ``Pure": not only the word\footnote{due either to Bourbaki or to P. Cohn \cite{C} in the 50's.} (which has so many acceptions) but the very notion has lapsed in Algebraic Geometry. In the triad {\it flat/faithfully flat/pure}, which expresses that the base change functor for modules is {\it exact/faithful exact/faithful}, the third term has always been the poor relative.  
 One may wonder why, since faithfully flat descent of modules works as well when ``faithfully flat" is replaced by ``pure" (surprisingly, {it is faithfulness, not exactness, of base change which is relevant here}). What is more, the canonical topology coincides with the (effective) descent topology.  

All this was probably well-known in the early 70's, but because the written records are either inaccessible or fragmentary, this has been largely forgotten, and {later rediscovered by pieces by different authors with different methods}. Since there is no complete account so far, we begin with ``untimely considerations" about the canonical topology on affine schemes, emphasizing its concrete description and its meaning for descent theory, {and we give a detailed exposition of what we believe to be Olivier's original approach}.  

We review module-splitting properties of pure ring maps (notably the fact that {pure extensions of a complete local Noetherian ring $R$ split} as $R$-modules)  
and discuss the question of whether the canonical topology on affine schemes is generated by the fpqc topology and by  morphisms $Y\to X$ such that $\cO_X\to \cO_Y$ 
splits as $\cO_X$-modules (Proposition \ref{rq2}). 

 A sufficient condition for $R\to S$ to be pure is the existence of an $S$-module which is faithfully flat over $R$ (Bourbaki). Ferrand asked about the converse, and we answer in the negative (Example \ref{exF}).

 \subsection{} In part II, we turn to our main topic: finite coverings. We discuss examples of finite coverings which separate the canonical, fpqc and fppf\footnote{fid\`element plate de pr\'esentation finie.} topologies from each other. {This} discussion puts in perspective the following key result: 
 
\smallskip {\it finite coverings of {\it regular} (Noetherian) schemes are coverings for the canonical topology, and even for the fpqc topology} (Theorems \ref{T1} and \ref{T2}). 
 
 \smallskip The first assertion is nothing but a {geometric translation} of the ``direct summand conjecture", while the second is a {geometric interpretation}
   of the existence of big Cohen-Macaulay algebras; both statements were conjectured by Hochster and proved by him in the presence of a base field,   and were proved in general (much later) by the first author, using perfectoid techniques \cite{A1}. On the other hand, finite coverings of regular Noetherian schemes need not be coverings for the fppf topology (Ex. \ref{ex 5.1}\footnote{{in the terminology of Commutative Algebra, the difference between fpqc and fppf here reflects the difference between big and small Cohen-Macaulay algebras.}}). 

 We establish a ``weakly functorial" version of this result (Theorem \ref{T3}), and discuss and exemplify its role. We also pay special attention to the case of the Frobenius morphism, and prove the following extension of Kunz' theorem: {\it in positive characteristic, a scheme is regular if and only if its Frobenius endomorphism is a covering map for the fpqc topology} (Theorem \ref{T6}).

 \subsection{} In part III, we fix the affine Noetherian scheme $X$ and vary the finite covering $Y\to X$.  Which schemes $X$ have the property that every finite covering is a covering for the canonical topology? 
 
 This question is the object of a very active area of Commutative Algebra, the theory of splinters, which we review in a geometric perspective. Beyond regular schemes, many quotient singularities are splinters (Theorem \ref{T5}). 
 
 There is a sharp contrast between splinters in characteristic $0$ (where splinter = normal)  and splinters in positive or mixed characteristic, where this notion happens to be much more restrictive; one might say, heuristically, that the canonical topology is ``much finer" in characteristic $0$ than in characteristic $p$.
  
 Our geometric viewpoint suggests the following analogous question: which schemes have the property that every finite covering is a covering for the fpqc topology? It seems that this question has not yet been explored (beyond the above result which tells that regular schemes have this property), and we give several partial answers.  
 
 \subsection{} We conclude that while the fpqc {\it pre}topology and its descent properties are well-known to every contemporary algebraic geometer, some aspects of the fpqc topology remain mysterious: which properties of schemes descend by (possibly non flat) morphisms which are coverings for the fpqc topology?  
 The case of regularity can be settled thanks to a recent theorem of Bhatt-Iyengar-Ma (Theorem \ref{T7}). 
   
 \subsection{} Outside of some new results and examples, the main purpose of this paper is to {\emph{synthesize}} a rather scattered theme. A large portion is expository, with a view of clarifying several old and new issues; at the same time, it is partly a work of translation and interpretation, recasting in  geometric terms some apparently remote notions and results from Commutative Algebra\footnote{``non tamquam transfuga, sed tamquam explorator" (Seneca, letter to  Lucilius n.2).}; a process which generates in turn some new questions.
  In front of a considerable technological inflation ($F$-singularities and derived avatars, perfectoids and prisms), we feel that such ``dictionaries" illustrated with examples may help the topic move from the esoteric to the exoteric.

  On a more personal level, it is from the geometric viewpoint advocated in this paper that the first author once learned from the second author about the direct summand conjecture, and ended up working in this fascinating area.

       \bigskip
{\footnotesize{\it Acknowledgements.}  We are grateful to D. Ferrand for confirming that our reconstruction of Olivier's proof matches with his memories, and for several other useful comments (cf. Remark~\ref{qF}). Thanks are also due to B. Bhatt for communicating some details about his very recent work on derived splinters in mixed characteristic, and for exchange about the arc-topology. 

This paper has been improved thanks to S.  Kov\'acs's remarks on rational singularities. We are also much indebted to L. Ma, who sent  
useful references and pointed out to us 
that his joint work with Bhatt and Iyengar provides a full answer to our question of
descent of regularity for coverings for the fpqc topology (cf. Theorem~\ref{T7}).
Finally we thank two referees for their careful reading and sharp remarks.
} 
 
\section*{\\Part I. Review of the canonical topology on affine schemes}

\section{{Review of the canonical topology on a category}}

\subsection{}\label{Subs:1.1}
Let us briefly recall the notion of {\it Grothendieck topology} on a category $\cC$. We assume for simplicity that $\cC$ has an initial object $\emptyset_{\cC}$ and that fiber products are representable. The Yoneda embedding $X \mapsto h_X = \Hom_{\cC}(\_,X)$ identifies $\cC$ with a full subcategory  of the category ${\cC}^\wedge$ of presheaves. 
 
Following \cite[II. 1.1]{SGA4}, a \emph{Grothendieck topology} $\cT$ on $\cC$ is the data, for every object $X$, 
 of a set $J_{\cT}(X)$ of sieves (= subobjects of $h_X$ in ${\cC}^\wedge$), called \emph{covering sieves}, containing $h_X$ and satisfying the usual base change and locality conditions.  

 These axioms imply that the intersection of two covering sieves is a covering sieve and that a sieve which contains a covering sieve is a covering sieve (this follows from locality \cite[0, 1.2.1]{Gir1})\footnote{in \cite[Def. 4.2.1]{SGA3} this was stated as an axiom but it follows from the axiom of locality since the base change of a sieve along any morphism in it is the maximal sieve.}. 
A presheaf ${\mathcal F}\in\cC^\wedge$ is a sheaf for $\cT$ if and only if for any object $X$ and any $R\in J_{\cT}(X)$ the canonical map
$F(X)\simeq \Hom_{{\cC}^\wedge}(h_X,F)\to \Hom_{{\cC}^\wedge}(R,F)$
 is bijective.

Fiber products being representable in $\cC$, we can
{translate the previous }
notion {in terms of} \emph{covering families of morphisms} instead of covering sieves 
(\cite[I. 4.3.3]{SGA4}, \cite[2.3.4]{V}), which is easier to handle. 
{Let us first recall that
a} \emph{Grothendieck pretopology} on $\cC$ is 
 the data for any $X$ in $\cC$ of  a class $\Cov_E(X)$ of families 
 of morphisms $(X_i\rightarrow X)_{i\in I}$
 called coverings of $X$ which are stable under base change and composition and such that:
 the family indexed on the empty set covers the initial object $\emptyset_{\cC}$ and
 any isomorphism $\{X'\rightarrow X\}$ is a covering of $X$.
 
 As in the case of topological spaces two different pretopologies can give rise to the same topology.
 {The notion of  Grothendieck topology in terms of covering sieves clarifies that, since any sieve containing a covering sieve is covering, a Grothendieck topology is a Grothendieck pretopology satisfying an extra axiom called the} 
  \emph{saturation axiom}:
any family ${\mathcal I}:=(X_i\rightarrow X)_{i\in I}$
 admitting families 
$(U_{j_i}\rightarrow X_i)_{{j_i}\in J_{i}}$  such that the {compositions}
$(U_{j_i}\rightarrow X)_{i\in I, j_i\in J_i}$ {form} a covering family of $X$,
 is a covering {family.}  
 {This axiom implies that }any split morphism $Y \rightarrow X$ is a covering for any Grothendieck topology since
it is refined by  $\id_X$ which is a covering of $X$.

\begin{remark}\label{R}
In terms of covering families, a
presheaf ${F}\in\cC^\wedge$ is a sheaf if and only if for any
covering family $(X_i\rightarrow X)_{i \in I}$
the diagram induced by the projections
$$\label{sessheaf}
\xymatrix{
F(X) \ar[r] & \prod\limits_{i\in I}F(X_i) 
\ar[r]\ar@<1ex>[r] & \prod\limits_{i, j \in I}
F(X_i\times_X X_j) \\
} $$is an equalizer (= is exact, \cite[II. 2.4]{SGA4}). By cofinality, this can be tested on any generating pretopology.\end{remark}

 \begin{remark}\label{R:empty} Some authors - and we agree - add the axiom that $\emptyset_{\cC^\wedge}\in J_{\cT}(\emptyset_{\cC})$: 
 ``the empty family covers the empty set", which ensures that for any sheaf $F(\emptyset)=\{pt\}$.
 This axiom prevents, for a given sheaf $F_0$, the proliferation of other sheaves $F$ indexed by an arbitrary pointed set $(S, s)$, which coincide with $F_0$ on any non empty object of $\cC$ but take value $F(\emptyset)=S$, the restriction maps $F(X)\to F(\emptyset)$ taking constant value $s$ (for $X$ non empty).   
\end{remark}

\subsection{} One can reverse the point of view, fixing first a family of (normalized) presheaves and looking for the finest topology such that they are sheaves. The answer is exposed in terms of covering sieves in \cite[II. 2.2]{SGA4}:

\begin{proposition}\label{toppsh}
Let ${\mathcal F}=({ F}_i)_{i\in I}$  be a family of presheaves on a category $\cC$ satisfying ${ F}_i(\emptyset_\cC)=\{pt\}$, and for any object $X$, let 
$J_{\mathcal F}(X)$ be the set of sieves $R$ of $X$ such that for any morphism $Y\rightarrow X$
the canonical map ${ F}_i(Y)\simeq \Hom_{\cC^\wedge}(R\times_XY, { F}_i)$ is bijective.
Then $J_{\mathcal F}(X)$ are the covering sieves for a Grothendieck topology, which is the finest for which every ${ F}_i$ is a sheaf. \qed
\end{proposition}

The \emph{canonical topology} can now be defined as  
the finest Grothendieck topology on $\cC$ such that all
representable presheaves are sheaves (note that $h_Y(\emptyset_\cC)=\{pt\}$). Topologies which are coarser than the canonical topology are called 
subcanonical.

For brevity, we shall write \emph{canonical covering} instead of covering for the canonical topology.

 For the canonical topology, the conditions of Proposition~\ref{toppsh} can be reformulated in terms of covering families  (\cite[IV. 4.3.9]{SGA3}):
 a family ${\mathcal U}=(X_i\rightarrow X)_{i\in I}$ is a canonical covering
 if and only if 
for any morphism $Z\rightarrow X$ in $\cC$,
the sequence
$$\xymatrix{
h_W(Z) \ar[r] & \prod\limits_{i\in I}h_W(X_i\times_XZ) 
\ar[r]\ar@<1ex>[r] & \prod\limits_{i, j \in I}
h_W(X_i\times_X X_j\times_XZ) \\
}
$$ 
 (induced by the family 
${\mathcal U}\times_XZ:=(X_i\times_XZ\rightarrow Z)_{i\in I}$)
is exact for every $W$ in $\cC$.  
Hence, in the case of a family given by a single morphism $Y\stackrel{f}\to X$,  by Yoneda, we obtain that $f$ is a   canonical covering if and only if for any morphism $Z\rightarrow X$,
$
\xymatrix{
Y\times_X Y\times_XZ \ar[r]\ar@<1ex>[r] &Y\times_XZ\ar[r] & Z
}
$
is a coequalizer 
(which is the definition of a \emph{universal strict epimorphism}).

\section{The canonical topology on affine schemes}

Let us recall the following basic definition 

\begin{definition}\label{Def:pure}
A ring map $R\stackrel{\alpha}\to S$  is \emph{pure} if it is 
universally injective, i.e. for any  commutative $R$-algebra $T$,  
$T\to S\otimes_R T$ is injective.
\end{definition}

We notice that any pure ring map $R\to S$ induces a  \emph{surjective}
map of spectra $\Spec S\to \Spec R$, because for any
point $x$ of $\Spec R$ with residue field $\kappa(x)$ the map
$\kappa(x)\to S\otimes_R \kappa(x)$ is injective (the converse fails as shown in 
Example~\ref{exbl}
below).

\subsection{} Let ${\bf{Aff}}_k$ be the category of affine schemes over a commutative ring $k$. 
A family ${\mathcal U}=(\Spec R_i\rightarrow \Spec R)_{i\in I}$ is a covering family
for the canonical topology in $\Aff$ if and only if for any commutative $R$-algebra $T$ the sequence 
$\xymatrix  {
T\ar[r] & 
\prod\limits_{i\in I}(R_i\otimes_RT)\ar@<0.5ex>[r]
\ar@<-0.5ex>[r] &{\prod\limits_{i, j \in I}}
(R_i\otimes_RR_j\otimes_RT)}$
is exact.

Let us remark that any ring map $R\stackrel{\alpha}\to S$ which corresponds to a canonical covering is pure.

{Canonical coverings are universal strict epimorphisms, 
but it turns out that ``strict"  follows from ``universal", as the following proposition shows; however, in order to see this, one has to leave the category of commutative rings and consider the category of modules.}

\begin{proposition}\label{P:5eq}
Let $R\stackrel{\alpha}\to S$ be a {ring map}. The following are equivalent:
\begin{enumerate}
\item $\alpha$ is a pure {ring map};
\item for any $R$-module $M$, the morphism $\alpha_M:=\alpha\otimes_R {\id_M}: M\hookrightarrow S\otimes_RM$ is a monomorphism;
 \item for any $R$-module $M$,  the sequence
$M\stackrel{\alpha_M}\to S\otimes_R M\rightrightarrows S\otimes_RS\otimes_R M$ is exact;
\item $\alpha$ is   a universal strict monomorphism, i.e.
for any ring map $R\to T$,
 the sequence
$T\stackrel{\alpha_T}\to S\otimes_R T\rightrightarrows S\otimes_RS\otimes_R T$ is exact;
\item the base change functor $R\lMod \stackrel{S\otimes_R\_}\to S\lMod$ is faithful. 
\end{enumerate}
\end{proposition}
\begin{proof}
$(1)\Rightarrow (2)$ In the case of an augmented algebra $S_M=R\times M$ with $M\cdot M=0$, purity
 implies that $\alpha\otimes_R \id_M: M\hookrightarrow S\otimes_RM$ is a monomorphism for any $M\in R\lMod$.
 
$(2)\Rightarrow (3)$ 
 In $(3)$, the double map is $( \eta_1={\alpha_{S\otimes_RM}}, \;\eta_2=1_S\otimes_R \alpha_M)$.
Let  $K\stackrel{\beta}\to S\otimes_RM$ be such that 
$\eta_1 \beta=\eta_2\beta$ and let $S\otimes_R M \stackrel{\pi}\to C$ be the cokernel of $\alpha_M$. One has a commutative diagram 
 {\small \[\xymatrixcolsep{4em}
\xymatrixrowsep{1.3em}
\xymatrix{
& K \ar[ldd]^(0.6)\beta|\hole\ar[dr]^{\beta}\ar@{..>}[ld]_\gamma& & \\ 
M\ar@{->}[rr]^{\alpha_M}\ar@{->}[d]_{\alpha_M}  & &
S \otimes_RM   \ar[d]^{1_S\otimes_R \alpha_M}  \\
S\otimes_RM\ar@{->}[rr]^{\alpha_{S\otimes_RM}}  \ar@{->}[d]_\pi& & S\otimes_R S\otimes_R M  \ar[d]^{1_S\otimes \pi} 
 \\
C\ar@{^(->}[rr]^{\alpha_{C}} & &S\otimes_R C .
\\}\]}
 One has $\alpha_C\pi\beta = (1_S\otimes \pi\alpha_M) \beta = 0$, hence $\pi\beta=0$ by $(2)$. Therefore 
there exists an unique map $K\stackrel{\gamma}\to M$ such that $\beta=\alpha_M\gamma$, which proves that  
$M$ is the equalizer of $(\eta_1, \eta_2)$.

$(3)\Rightarrow (4)$
Condition $(4)$ is a 
 special case of $(3)$ when $M$ is an $R$-algebra.
 
$(4)\Rightarrow (1)$ This is clear, since pure means universally injective.

$(2) \Rightarrow (5)$ Let $M\stackrel{\phi}\to N $ satisfy $1_S \otimes \phi= 0$. Since $\alpha_N\phi = (1_S\otimes \phi)\alpha_M$ and $\alpha_N$ is injective by $(2)$, one has $\phi=0$.

$(5)\Rightarrow (2)$ Let $K\stackrel{\psi}\to M$ satisfy $\alpha_M \psi=0$. By $S$-linearity, $1_S \otimes \psi =0$, and by $(5)$, $\psi =0$, which shows that $\alpha_M $ is injective.  
\end{proof}

 \begin{remark} In the literature of the 60's, for any extension $R\stackrel{\alpha}\hookrightarrow S$, the kernel of the double map  $S \rightrightarrows S\otimes_R S$ was coined the ``dominion" of $\alpha$ (many  forgotten nice results were proven about dominions). The construction can be iterated, replacing $S$ by the dominion and so on, even transfinitely; at the end, one gets a factorization $R\to {\rm{Coim}}\,\alpha\to S$, where $R\to {\rm{Coim}}\,\alpha$ is the largest epimorphism with target contained in $S$  (the name ``coimage" of $\alpha$ was suggested by Grothendieck, cf.  \cite{I}). 
 \end{remark}
 
 \begin{remark}\label{qF} In \cite[I, \S 3, 5]{Bou}, it is shown that  $R\to S$ is pure if there exists an $S$-module which is faithfully flat over $R$. The {\it question of the converse} has been raised by Ferrand, and we shall answer it in the negative at the end of the paper (Example \ref{exF}).   
   \end{remark}
   
\subsection{}  If {the map} $\alpha$ splits in $R\lMod$, it is pure. The converse holds if {the quotient} $S/R$ is of finite presentation \cite[I.2]{La}. This fact is fundamental in the sequel, especially the following special case:
 
  \begin{proposition}\label{P:2} A \emph{finite} algebra $S$ over a Noetherian ring $R$ is pure if and only if it splits in $R\lMod$.  \end{proposition}
  \begin{proof}
Let $R\stackrel{\alpha}\to S$ be a pure ring map. Given a morphism $R^n\stackrel{\delta}\to R^m$, one has a commutative diagram
$$\small\xymatrixcolsep{3em}
\xymatrixrowsep{1.5em}
\xymatrix{ 
R^n\ar[r]^{\delta} \ar[d]_{\alpha_{R^n}}& R^m \ar[d]^{\alpha_{R^m}} \ar[r] & \Coker \delta \ar[r] 
\ar@{_(->}[d]^{\alpha_{\Coker \delta}} & 0 \\
S^n \ar[r]^{1_S\otimes_R\delta} & S^m\ar[r] & S\otimes_R\Coker \delta \ar[r] & 0.\\ }$$ Since $\alpha_{\Coker \delta}$ is injective by purity, it follows that for any pair $\psi\in R^m$, $\phi\in S^n$ such that $\alpha_{R^m}(\psi)=(1_S\otimes_R\delta)\phi$,
there exists $\epsilon\in R^n$ such that $\delta(\epsilon)=\psi$.

Now assume that there is a finite presentation $R^m\stackrel{\rho}\to R^n\stackrel{\pi}\twoheadrightarrow {S/R}$. One can find
 $\psi$ and $\phi$ which make the following diagram commute
$$\small\xymatrixcolsep{3em}
\xymatrixrowsep{1.5em}
\xymatrix{ 
0\ar[r] & R \ar[r]^{\alpha} &S\ar[r]^{\nu} & {S/R}\ar[r] \ar@{=}[d]& 0\\
 & R^m \ar[r]^{\rho} \ar[u]^\psi &R^n\ar[r]^\pi\ar[u]^\phi & {S/R}\ar[r] & 0.\\
}$$
Applying the previous observation to $\delta=\rho^t: \Hom_R(R^n,R)\to \Hom_R(R^m,R)$
(since $(1_S\otimes_R\delta)\phi=\phi\rho=\alpha\psi=\alpha_{R^m}(\psi)$), one deduces that there exists
$R^n\stackrel{\epsilon}\to R$ such that $\epsilon\rho=\psi$.  Then $\tau:=\phi-\alpha\epsilon$ {satisfies} $\tau\rho=0$, hence factors through a section of $S\to S/R$.
   \end{proof}

{ \begin{remark}\label{r 1}  Generalizing Definition~\ref{Def:pure}, a morphism of $R$-modules $M\stackrel{\phi}\to N$  is said to be pure if it remains injective after tensoring with any $R$-module. Expressing $ \Coker\, \phi$ as a filtered colimit of $R$-modules of finite presentation, the same argument shows that $\phi$ is pure if and only if it is a filtered colimit of split {monomorphisms} $M\stackrel{\phi_i}\to N_i$ \cite[2.3]{La}.  

  We refer to \cite[I.2]{La}, \cite{Ol}, 
  \cite[2.1]{HH2}, \cite{Do},
  \cite[5]{P3}, \cite[A]{A1}, \cite[08WE]{St} for more on pure morphisms. 
\end{remark} }

There is another general situation where pure ring maps split:
\begin{proposition}\label{Prop:2+}\cite[Lem. 1.2.]{Fed}\footnote{This result does not seem to be widely known. The reference to 	\cite{Fed} was pointed out to us
by two referees independently and L. Ma indicated to us that 
the argument also appears  in the online course \cite[p. 155]{H4}; on p. 84 of this course, it is also shown that if $S/R$ is finitely presented the splitting property is fpqc-local, so that Proposition \ref{P:2} can be derived, after localization/completion, from Proposition \ref{P:Rlocs}.}\label{P:Rlocs} An algebra $S$ over a \emph{complete local} Noetherian ring $R$ is pure if and only if it splits in $R\lMod$. 
 \end{proposition}
\begin{proof}
Let $E$ be the injective hull of the residue field of $R$. 
Since $R$ is complete local Noetherian, it is naturally isomorphic to $\Hom_R(E,E)$ (cf. e.g. \cite[47.7.5]{St}).
By the purity of $\alpha$,  $E\stackrel{\alpha_E}\hookrightarrow S\otimes_RE$ is injective, and since $E$ is an injective $R$-module,
there exists a morphism of $R$-modules $S\otimes_RE\stackrel{\beta_E}\to E$ which makes the triangle 
 {\small\[\xymatrixcolsep{3em}
\xymatrixrowsep{1.5em}
\xymatrix{
E\,\ar@{^(->}[r]^(0.4){\alpha_E}\ar[d]_{\id_E}& S\otimes_RE \ar@{-->}[ld]^{\beta_E} \\ 
E & \\ }\]} 
commute.
The $R$-linear morphism $\beta:S \to \Hom_R(E,E)\simeq R$ defined by $\beta(s):=\beta_E(s\otimes\_)$ is a retraction of $\alpha$,
since $\beta(1_S)(x)=\beta_E\alpha_E(x)=x$ for any $x\in E$.
\end{proof}

{\begin{remark} When $Y\stackrel{f}{\to} X$ is a finite canonical covering of a {\emph{non affine}} Noetherian scheme $X$,  $\cO_X \to f_\ast \cO_Y$ does not necessarily split (even if $f$ is flat). A counterexample is given by any abelian variety $X=Y$ of positive dimension in characteristic $p>0$ and $f = [p]_X$ is the multiplication by $p$ in the group scheme $X$: the map $ \cO_X \to [p]_{X \ast }\cO_X$ vanishes on $H^1$  \cite[Ex. 2.11]{B1}. This is one 
{\emph{main reason}} for which we have chosen to focus on affine schemes in this paper.
\end{remark}
} 

We shall also need the following lemma, which is announced by Olivier in \cite{Ol}:
 
 \begin{lemma} 
  \label{Lemma:1}
  Let $R\stackrel{\alpha}\to S$ be a pure ring map, and $M\stackrel{\varphi}\to N$ be a morphism of $R$-modules. Then ${\rm{Im}}\, \varphi \to  {\rm{Im}}\, (1_S\otimes \varphi )$  is a pure morphism of $R$-modules. 
 \end{lemma} 
 
 \begin{proof} By Remark \ref{r 1}, $\alpha $ is a filtered colimit of  $R$-linear maps $R \stackrel{\alpha_i}\to S_i$ which admit a splitting $\sigma_i$.  
 The image of the composed map $\sigma_i\otimes \varphi: {\rm{Im}}\, (\id_{S_i}\otimes \varphi)  \to S_i\otimes_R N \stackrel{\sigma_i\otimes \id_N}\to N $ lies in ${\rm{Im}}\, \varphi$, hence the splitting $\sigma_i\otimes \id_N$ induces a splitting of ${\rm{Im}}\, \varphi\to {\rm{Im}}\, (\id_{S_i}\otimes \varphi) $. Passing to the colimit, ${\rm{Im}}\, \varphi \to  {\rm{Im}}\,( 1_S\otimes \varphi )$  is pure.
 \end{proof}

\subsection{} Since flatness (resp. faithful flatness) corresponds to exactness (resp. faithful exactness) of the base change functor, the characterization $(5)$ of purity {(Proposition~\ref{P:5eq})} shows that ``faithfully flat = pure + flat". This immediately shows, in geometric terms, the following: 

\begin{lemma}\label{L1} If $Z\stackrel{g}\to Y\stackrel{f}\to X$ are morphisms in ${\bf{Aff}}_k$, such that $g$ is a canonical covering and $fg$ is  flat (resp. faithfully flat), then $f$ is flat (resp. faithfully flat). \qed
\end{lemma}

Let us also mention the following lemma from \cite[2.2]{HH2} (here reformulated in terms of canonical coverings), which is useful {for reduction} to the local case:

  \begin{lemma}\label{L2} For any canonical covering $Y\to X$, with $X$ local, there exists a closed point $y\in Y$ such that $\Spec\, \cO_{Y,y}\to X$ is a canonical covering.  \qed \end{lemma} 
 
   We end this section with a typical example of a surjective affine morphism which is not a canonical covering.
 
\begin{example}\label{exbl} Let $Bl_O \mathbb A^2_k \to X = \mathbb A^2_k$ be the blow-up of the origin and let $Y$ be the disjoint union of its two standard affine charts. Then $Y\stackrel{f}\to X$ is not a canonical covering in ${\bf{Aff}}_k$. 
 Indeed, the diagonal homomorphism $R =k[x,y] \to S =k[x,y/x] \times k[x/y, y]$ is not pure: modulo the ideal $(x^2, y^2)$, it sends the nonzero class of $xy$ to the class of $ (x^2 \cdot (y/x), (x/y)\cdot y^2)$ which is $0$. In other words, $xy\notin (x^2,y^2)$ but $xy\in (x^2, y^2)S\cap R$. 
  \end{example}

\section{The (effective) descent topology on affine schemes}

\subsection{} Let  $R\stackrel{\alpha}\to S$ be a {ring map} and let $N$ be an $S$-module.
We denote by
$S\otimes_RN$ and $N\otimes_RS$ the two different $S\otimes_RS$-module structures 
on the tensor product of $N$ with $S$ given by
$(s_1\otimes s_2)(n\otimes t)=s_1n\otimes s_2t$ and
$(s_1\otimes s_2)(t\otimes n)=s_1t\otimes s_2n$ respectively.
Given an $S\otimes_RS$-linear isomorphism
$N\otimes_RS \stackrel{\phi}\to S\otimes_RN$, we denote by 
$\phi_i$ the natural morphism obtained by tensoring $\phi$ with $1_S$ 
in the $i$-th position for $i\in\{1,2,3\}$.
For example, 
$\phi_2=({1_S}\otimes \alpha\otimes \id_N)\circ \phi$; in formula, if
$\phi(n\otimes s)=\sum_{j=1}^k s_j\otimes n_j$, then
 $\phi_2(n\otimes s)=\sum_{j=1}^k s_j\otimes 1_S\otimes n_j$.

{
A descent datum on $N$ is an $S\otimes_RS$-linear isomorphism
$N\otimes_RS \stackrel{\phi}\to S\otimes_RN$ such that 
$\phi_1\phi_3=\phi_2$ and $\lambda_N\theta_N=  \id_N$, where
$\lambda_N:S\otimes_RN\to N$ is {the} external multiplication on $N$ and
$\theta_N(n):=\phi(n\otimes 1_S)$.
}

Let  $DD(\alpha)$ be the category whose objects are 
$S$-modules endowed with a descent datum, and 
morphisms are  $S$-linear maps compatible with the descent data.
 
The functor $S\otimes_R\_$ factors through $DD(\alpha)$ (since any $S\otimes_RM$ admits a natural descent datum).
  In the following commutative diagram of functors, $C_\alpha$ is called the comparison functor while
For is the forgetful functor:
{\small\[\xymatrixrowsep{1.5em}
\xymatrix@-5pt{
R\lMod \ar[rr]^{S\otimes_R\_} \ar[rd]_{C_\alpha}
& & S\lMod. \\
& DD(\alpha) \ar[ru]_{{\rm{For}}}& \\
}
\]}

\begin{lemma}\label{Lemma:2}
Given a ring map $R\stackrel{\alpha}\to S$ and $(N,\phi)$ an object in $DD(\alpha)$, 
the diagram
\begin{equation}\label{E:seN}
\xymatrixcolsep{4em}
\xymatrix{
N\ar[r]^(0.4){\theta_N}  & 
  S\otimes_RN \ar@<0.5ex>[r]^(0.45){1_S\otimes \alpha_N}
\ar@<-0.5ex>[r]_(0.45){1_S\otimes \theta_N}  &S\otimes S\otimes N \\
}
\end{equation}
is a split equalizer.  The sequence 
\begin{equation}\label{E:seN1}
\xymatrix{
0\ar[r] & N\ar[r]^(0.4){\theta_N}  &S\otimes_RN \ar[r]
&S\otimes_R \Im(\alpha_N-\theta_N)\ar[r] &0 \\ }
\end{equation}
is split exact  in $S\lMod$.
\end{lemma}
\begin{proof}
We prove that \eqref{E:seN} is a split equalizer with $S$-linear retractions $\lambda_N$
 and $\mu_S\otimes \id_N$ respectively (where $\mu_S$ is the multiplication on $S$), that is:
\begin{enumerate}
 \item $(1_S\otimes \theta_N) \theta_N=(1_S\otimes \alpha_N) \theta_N$;
\item $\lambda_N  \theta_N=\id_N$;
 \item $(\mu_S\otimes \id_N)(1_S\otimes \alpha_N)=\id_{(S\otimes_RN)}$;
 \item $ (\mu_S\otimes \id_N)(1_S\otimes \theta_N)=\theta_N\lambda_N $.
\end{enumerate} 
 Equality $(1)$ is the cocycle condition $\phi_3\phi_1=\phi_2$ restricted to $N$. $(2)$ holds by definition of a descent datum.
For $(3)$, note that  $$(\mu_S\otimes_R \id_N)(1_S\otimes_R \alpha_N)(s\otimes n)= 
(\mu_S\otimes_R \id_N)(s\otimes 1_S\otimes n)=s\otimes n.$$
For $(4)$, let us set  
 $\theta_N(n)=\sum\limits_{k=1}^m s_k\otimes n_k$. Since $\phi$ is $S\otimes S$-linear, $\theta_N(sn)=\sum\limits_{k=1}^m ss_k\otimes n_k$, so that
$ 
 (\mu_S\otimes \id_N)(1_S\otimes \theta_N)(s\otimes n)
=(\mu_S\otimes \id_N)(\sum\limits_{k=1}^m s\otimes s_k\otimes n_k)=
\theta_N( sn). $ 

 For the second assertion, we observe that $S\otimes_RN \to 
 S\otimes_R \Im(\alpha_N-\theta_N)$ is split, with $S$-linear retraction given by the composition of the natural map $S\otimes_R \Im(\alpha_N-\theta_N)\to S\otimes_RS\otimes_RN$ and $\mu_S\otimes \id_N$: this follows from 
  $$(5)\; \; (1_S\otimes (\alpha_N-\theta_N))(\mu_S\otimes \id_N) (1_S\otimes (\alpha_N-\theta_N))=  1_S\otimes (\alpha_N-\theta_N),$$ which is a combination of $(3), (4)$ and $(1)$.
   \end{proof}

\subsection{}  A {ring map} $R\stackrel{\alpha}\to S$ is called an \emph{effective descent} (resp. descent)  morphism
if the functor $C_\alpha$ is an equivalence of categories (resp. a fully faithful functor).
In the same way a family 
  of morphisms $(X_i\rightarrow X)_{i \in I}$ in ${\Aff}$ is called an effective descent family 
 if the analogous comparison functor is an equivalence of categories.
 In \cite[II, Prop. 1.1.3]{Gir1}, Giraud proved that effective descent families are the covering families for a
 Grothendieck topology called the {\emph{effective descent topology}}.

\begin{theorem}\label{Tdesc} The canonical topology on ${\Aff}$ coincides with the effective descent topology. 
\end{theorem}

For simplicity,  we treat here the case of a covering given by a single ring map  $R\stackrel{\alpha}\to S$ 
 (this can be adapted to the case of a general family, taking products on the index of the family).
This case is covered by the following result (which is based on Proposition~\ref{P:5eq}
{and Lemma~\ref{Lemma:1}}).

 \begin{proposition}[Olivier]\label{descent}
 Let $R\stackrel{\alpha}\to S$ be a {ring map}. The following are equivalent:
 \begin{enumerate}
\item $\alpha$ is an effective descent morphism;
\item $\alpha$ is a descent morphism;
\item $\alpha$ is a pure {ring map}.
\end{enumerate}
 \end{proposition}
 
 In his note \cite{Ol} (which is  just an announcement),
 Olivier comments that the key point is the intermediate result that we have restated as 
 Lemma~\ref{Lemma:1}.  Following this hint, we have tried to reconstitute Olivier's proof
 (there are other proofs and sketches of proof in the literature, notably by Mesablishvili \cite{M1}).

 \begin{proof}
 $(1)\Rightarrow (2)$ is clear.
 
$(2)\Leftrightarrow (3)$.
 Let $R\stackrel{\alpha}\to S$  be a descent morphism.
 The comparison functor $C_\alpha$
is then fully faithful, which amounts to: for any pair of $R$-modules $M_1, M_2$, 
 \begin{equation}\label{E:cf}\small{
\xymatrixcolsep{1em}
\xymatrixrowsep{1em}
\xymatrix{
\Hom_R(M_1,M_2)\ar[r] \ar@{=}[d]& 
\Hom_S(S\otimes_RM_1,S\otimes_RM_2)\ar@<0.5ex>[r]
\ar@<-0.5ex>[r] \ar@{}[d]|*[@]{\cong}&\Hom_{S}(S\otimes_RM_1,S\otimes_RS\otimes_R M_2)\ar@{}[d]|*[@]{\cong}\\
\Hom_R(M_1,M_2)\ar[r] &
\Hom_R(M_1,S\otimes_RM_2)\ar@<0.5ex>[r]\ar@<-0.5ex>[r] &
\Hom_{R}(M_1,S\otimes_RS\otimes_R M_2)}
}\end{equation}
  is an exact sequence. Considering $M_2$  as fixed and letting $M_1$ vary, we get by Yoneda 
 that $M_2\to S\otimes_R M_2\rightrightarrows S\otimes_RS\otimes_R M_2$   is an exact sequence. Since this holds for any $M_2$,  $\alpha$ is pure
by Proposition~\ref{P:5eq}.
Conversely, if $\alpha$ is pure, \eqref{E:cf}  is an exact sequence.

 $(3)\Rightarrow (1)$. Let $R\stackrel{\alpha}\to S$ be a pure ring map.
 Since $(3)$ implies $(2)$ the functor $C_\alpha$ is fully faithful, it remains to prove that it is essentially surjective.
Given an object $(N,\phi)$ in $DD(\alpha)$, 
 we set $M:=\Ker(\alpha_N-\theta_N)$ and denote by $\varphi$ the inclusion of $M$ in $N$. 
 Let us consider the diagram 
 {\small \begin{equation*} 
 {\xymatrixcolsep{1em}
\xymatrixrowsep{1em}
\xymatrix{0\ar[r] & N\ar[r]^(0.4){\theta_N}  &S\otimes_RN \ar[r]
&S\otimes_R \Im(\alpha_N-\theta_N)\ar[r] &0 \\ &S\otimes_R M \ar@{..>}[u]^\rho\ar[ur]_{1_S\otimes \varphi}& }
}\end{equation*}}
Since the row is exact by Lemma ~\ref{Lemma:2}, there is a unique morphism $\rho$ of $S$-modules which makes the diagram commute, and $\rho $ is   
a morphism in $DD(\alpha)$.
Let us show that it is an isomorphism. 
By exactness of the row, and right-exactness of $\_\otimes_R\_$, $\rho$ is surjective, and $\theta_N$ identifies $N$ with the image of $1_S\otimes \varphi$ (in such a way that the composition $\Im \,\varphi \hookrightarrow N \stackrel{\cong}\to \Im\,(1_S\otimes \varphi)$ is the natural map). Since $\alpha$ is pure, we conclude by Lemma~\ref{Lemma:1} that $ \varphi$ is  pure.  
 Therefore $1_S{\otimes} \varphi$ is injective, and since it factors through $\rho$, $\rho$ is also injective.  \end{proof} 

 {\footnotesize\subsection{}  Let us discuss the relationship between the canonical topology, which coincides with the universal effective descent topology,  and some other Grothendieck topologies on schemes, old or new, which are also related to descent properties. We restrict ourselves to ${\Aff}$ for the purpose of coherence.

A morphism $Y\to X$ in ${\Aff}$ is a {\it v-covering} if any valuation ring of $X$ lifts to a valuation ring of $Y$ \cite{BS1}. If ``valuation" is replaced by ``valuation of rank one", one gets the notion {\it arc-covering} \cite{BM}. The corresponding Grothendieck pretopologies on ${\Aff}$ are topologies, i.e. satisfy the saturation axiom of \ref{Subs:1.1}.
The $v$-topology plays a role in the ongoing geometrization of the $p$-adic Langlands program, while the descent properties (excision) of the arc-topology allow to simplify and unify several deep results in \'etale cohomology, and play a role in the theory of prisms. They coincide in the noetherian context, but not in general \cite[2.21]{BM}. 
Both admit interpretations in terms of classical topologies: 

the $v$-topology coincides with the ``universally substrusive" topology, which was already investigated by Picavet in the 80's, cf. \cite{P2}\cite[8]{P3} (revisited by Rydh \cite{Ry} in the non Noetherian situation);  
  the (finer) arc-topology coincides with the ``universally spectrally submersive" topology, cf. \cite[2.2]{BM}.  
 
  \medskip Old work by Olivier shows that pure morphisms induce the quotient topology at the level of spectra, hence shows (beforehand) that the arc-topology and the $v$-topology  are finer than the canonical topology on ${\Aff}$ (strictly finer: example \ref{exbl} provides a $v$-covering \cite[Ex. 5]{P3} which is not a canonical covering). 
  
  On the other hand, if $k={\bf F}_p$, the canonical topology on the category of {\it perfect} schemes in ${\Aff}$ coincides with the restriction of the arc-topology on ${\Aff}$ \cite[5.16]{BM}. It is strictly finer than the restriction of the canonical topology on ${\Aff}$: if $V$ is a perfect valuation ring of rank two, $x$ is a nonzero element of the height one prime $\frak p$ and $y$ is an element of the maximal ideal not in $\frak p$, then the diagonal map $ V\to V_{\frak p}\times V/{\frak p}$ is not pure (modulo $(xy)$, it sends the nonzero class of $x$ to $0$) but induces an arc-covering \cite[2.9]{BM}. }

\section*{\\ Part II. Finite coverings with regard to the canonical, fpqc and fppf topologies}

Let as before ${\bf{Aff}}_k$ be the category of affine schemes over a commutative ring $k$. 
\smallskip

By a \emph{finite covering}\footnote{of course, ``finite covering" does not refer to a covering family for some topology with finitely many morphisms (``finite cover" would be a better name for this notion). Here, the morphisms themselves have to be finite - and in the affine case, one reduces to 
{the} case of a single such morphism.}  of an affine $k$-scheme $X$, we mean a finite surjective morphism $ Y \stackrel{f}{\to} X$ in ${\bf{Aff}}_k$.

\smallskip

If $X  $ is Noetherian, any finite pure extension of $\cO(X)$ splits, so that {\it  a finite covering $Y\to X$ is a covering for the canonical topology (we say: {canonical finite covering}) if and only if  $\cO(X) \to \cO(Y)$ splits as a morphism of $\cO(X)$-modules}.

A morphism $Y\stackrel{f}{\to} X$ in ${\bf{Aff}}_k$ is {\it a covering for the fpqc topology (resp. fppf topology) if and only if there is an $\cO(Y)$-algebra which is faithfully flat over $\cO(X)$ (resp. and of finite presentation)}. Of course, $f$ need not be flat (resp. nor of finite presentation).

\section{Finite coverings which are not coverings for the canonical topology} 

\subsection{} The simplest such examples $ Y  \stackrel{f}{\to} X$ are the following three ones.

\begin{example}\label{ex 3.1} $Y= X_{red} \stackrel{f}{\to} X$, for any non reduced $X$, is not a canonical covering. Indeed, $f$ is not an epimorphism in ${\bf{Aff}}_k$. 

(On the other hand, if $Y \stackrel{f}\to X$ is a canonical covering with $X$ reduced, one may ask whether $Y_{red} \to X$ is still a canonical covering. This may fail
even if $X$ and $Y$ are Noetherian\footnote{Ferrand and Raynaud \cite{FR} have constructed a non normal Noetherian local domain $R$, with normalization $R_{nor}$ such that $\hat R_{red} =\widehat{ R_{nor}}$; thus $R\to \hat R$ is faithfully flat but $R\to \hat R_{red} $ is not pure.}).  
 \end{example}

\begin{example}\label{ex 3.2} The normalization $Y=X_{nor} \stackrel{f}{\to}  X= \Spec\, k[x,y]/(xy)$ of two crossing lines is not a canonical covering. Indeed,  $k[x,y]/(xy) \to k[x] \times k[y]$ is not a pure ring map since its reduction modulo $(x-y)$ sends the nonzero class of $x$ to $0$. Here $f$ is a non universal strict
epimorphism. \end{example}

\begin{example}\label{ex 3.3} More generally, the normalization $Y=X_{nor} \stackrel{f}{\to}  X = \Spec\, R$ of any non normal reduced excellent affine Noetherian scheme is a finite covering which is not a covering for the canonical topology. Indeed, if a fraction $y/x $ belongs to the integral closure $S$ of $R$, $y$ belongs to $xS\cap R $ which is $xR$ if $R\to S$ is pure, and in that case, $y/x\in R$.  

For instance, $ k[t^3, t^5]\to k[t]$ is a non pure finite extension; what is more, the corresponding morphism $f$ 
is a non strict
 epimorphism (\cite[V. 2.b]{SGA3}).
 \end{example}

\subsection{} Let us turn to examples of non canonical finite coverings $Y \stackrel{f}{\to}  X  $ with $X$ {normal}. There is no such (excellent) example if $k$ is a field of characteristic $0$: indeed, one can replace $Y $ by its normalization (which is still finite over $X$) and assume that $X$ and $Y$ are both integral; taking the trace and dividing by the degree gives {{a splitting of $ \cO(X)\to \cO(Y)$, so that $f$ is a canonical covering.}}

In positive or mixed characteristic, one cannot divide by the degree, and {{there are indeed examples}} of non canonical finite coverings with a {\it normal} base $X$.  In the discussion, we shall use the following lemma.

\begin{lemma}\label{l4.4} In the Noetherian situation, finite canonical coverings descend the property of being Cohen-Macaulay\footnote{on the other hand, the example of the double covering of the standard quadric cone (studied below \ref{ex 4.7}) shows that finite canonical coverings do not descend regularity.}.
\end{lemma} 

\begin{proof}  
Assuming that any secant sequence\footnote{\label{F1}i.e. a sequence $\underline x$ such that dimension drops by $1$ when one successively divides out by $x_i$; in another language, it is called part of a system of parameters.} in $S$ is regular
or, equivalently, weakly regular\footnote{i.e. such that the image of $x_{i+1}$ is a nonzero divisor after dividing out by $x_1, \ldots, x_i$.}  (since $S$ is Noetherian),
we have to show that  the same is true for $R$. Since $R\to S$ is finite, any secant sequence in $R$ remains secant in $S$. Since $R\to S$ is pure, any secant sequence  $\underline x$ in $R$ which becomes weakly regular in $S$ is already weakly regular in $R$.
Indeed, in the following diagram, the left vertical  and the top horizontal arrows 
are injective since $\alpha$ is pure and $S$ is Cohen-Macaulay
\[
\xymatrix{
{S\over (x_1,\cdots, x_i )}\ar@{^(->}[r] & {S\over (x_1,\cdots, x_i) } \\
{R\over (x_1,\cdots, x_i )}\ar[r] \ar@{^(->}[u] & {R\over (x_1,\cdots, x_i) }. \ar@{^(->}[u] \\
}
\]\end{proof} 

 \begin{example}\label{ex 3.4} The first published  ``normal example'' was $Y = \mathbb A^4_k \stackrel{f}{\to} X =  \mathbb A^4_k/(\mathbb Z/4\mathbb Z)$, where $\mathbb Z/4\mathbb Z$ acts on $\mathbb A^4_k$ by cyclic permutation of the variables, and $k$ is a field of characteristic $2$. It was proved in \cite{Be} that $X$ is normal  
 but not Cohen-Macaulay. By the previous lemma, $f$ cannot be a covering for the canonical topology. 
    \end{example}
 
  \subsection{} However, there is no hope to find an example of a non canonical finite covering of a  {\it regular} base $X$:

\begin{theorem}\label{T1}\cite{A1} Any finite covering of a regular Noetherian scheme is a covering for the canonical topology.\qed
\end{theorem}

This reduces immediately to the affine situation, and since a finite pure {map} of Noetherian rings splits {(as modules)}, the theorem is {\it equivalent} to Hochster's {\it direct summand conjecture}, proved by Hochster if $k$ is a field \cite{H1}, and by the first author in general \cite{A1}, namely: {\it every finite (commutative) algebra over a regular ring $R$ splits as an $R$-module}.  Theorem \ref{T1} is the {geometric translation} of this (ex-)conjecture. 

In  case $k$ is a field of characteristic $0$, as mentioned above, this is settled by a simple argument of divided trace. In case $k$ is a field of characteristic $p$, Hochster gave several short proofs which all rely on a clever use of the Frobenius morphism. In  mixed characteristic, the first author's proof uses perfectoid techniques.
 
\begin{remark}\label{r 3.6} According to Proposition \ref{descent}, an equivalent form of Theorem \ref{T1} is that any finite covering of a regular affine Noetherian scheme is an effective descent morphism for modules.
 \end{remark}

 \begin{remark} Theorem \ref{T1} also holds when $f$ is an {\it integral {surjective} morphism}:  one reduces to the case where $X$ is affine, then $f$ is a filtered colimit of finite coverings, and the statement comes from Theorem \ref{T1}.
  \end{remark}

  \begin{remark} If $X$ is affine regular, then for any finite covering $Y$ of $X$, $\cO(Y)$ is an $\cO(X)$-module of finite projective dimension. Koh asked whether, without assuming that $X$ is regular, any finite covering $Y\to X$ such that $\cO(Y)$ is of finite projective dimension over $\cO(X)$ is a canonical covering. This is true in characteristic $0$, but may fail otherwise \cite[\S 2]{Ve}.
 \end{remark}
 
{ \begin{remark} We have seen that canonical coverings are strict, and that finite coverings of normal affine schemes are canonical in characteristic $0$. In fact, every finite covering of a normal affine scheme is strict. After passing to connected components, we are reduced to show that if $R\to S$ is a finite extension of an integrally closed ring, $R$ is the equalizer of the double map $S\rightrightarrows S\otimes_R S$. Indeed, this is clearly true at the level of vector spaces after tensoring with the fraction field $Q(R)$, and the assertion follows from the fact that $S \cap Q(R)= R$ since $R$ is integrally closed.    \end{remark}}
    
\subsection{Descent of flatness.} It is well-known (and a formal consequence of Proposition \ref{P:5eq}) that pure ring maps $R\to S$ descend flatness of modules. It is more difficult to see that finite monomorphisms descend flatness  \cite{F} \cite[II, 1.2.4]{RG}. Integral monomorphisms descend flatness of finite modules \cite{BR}. If $R$ is not Noetherian, this is not true for arbitrary modules \cite[p. 121]{La}. 
 If $R$ is Noetherian, integral monomorphisms descend flatness of modules: this statement turns out, surprisingly, to be  another {\it equivalent form of Theorem \ref{T1}} \cite{O} (one implication was already implicit in \cite[II.1.4.3.2]{RG}). { See \cite{P2} for more on descent of flatness.}

 \subsection{Frobenius.}\label{frob}  Let $k$ be a field of characteristic $p>0$. Raising coordinates to the power $p$ induces an endomorphism $F_X$ of any $X$ in $\textbf{Aff}_k$: {in fact, it comes from an endomorphism $F$ of the identity functor of $\textbf{Aff}_k$. }

It is often the case that $F_X$ is a finite covering, for instance when $X$ is of finite type over a perfect field $k$ or $X$ is the spectrum of a complete local algebra over a perfect field. 
  
The condition that $F_X$ is a canonical covering - equivalently, that $\cO_X \to F_{X\ast}\cO_X $ splits, when $F_X$ is finite or in the local complete situation (see Propositions~\ref{P:2}, \ref{P:Rlocs}) - has been intensively investigated since the work of Hochster and Roberts  \cite{HR}
 and has led to the theory of $F$-singularities (cf. \cite{TW} for a survey). {For instance, the Frobenius morphism of the quadric cone in odd characteristic (cf. \ref{ex 4.4}) is a canonical covering.} {It turns out that so-called $F$-pure singularities} are  
 characteristic $p$ analogs of log-canonical singularities in characteristic $0$.  They need not be normal nor Cohen-Macaulay.

\section{Finite coverings which are coverings for the canonical topology but not for the fpqc topology}

\subsection{Frobenius-type extensions}\label{frobalg}

Let $R$ be a (commutative) Noetherian ring, and $M$ be a finitely generated $R$-module.  Let $b$ be a symmetric bilinear form 
$  S^2 M \to R$  satisfying the ``associativity rule"
$$ \forall \ell, m, n\in M, \; \ell \,b(m\cdot n) = b(\ell\cdot m)\, n.$$
  A straightforward computation (standard in the context of symmetric Frobenius algebras) shows that the multiplication rule 
 $$(r, m). (r',m') := (rr' + b(m\cdot m') , rm'+r'm)$$
 turns $R\times M$ into a commutative $R$-algebra $S$.  
 
 We set $X= \Spec \, R, \,Y= \Spec\, S.$

 We set $I  = {\rm{Ann}}_R \,M, \; J  = {\rm{Ann}}_R \,I.$ Note that the image of the quadratic form $q$ attached to $b$ is killed by $I$, hence lies in $J$.

\begin{proposition} Assume that the image of $q$ does not lie in $J^2$. Then the finite covering $Y\to X$ is a covering for the canonical topology but not for the fpqc topology. 
\end{proposition}

 \proof Obviously, $R\to S$ is pure since $R$ is a direct summand. 
 
  If $Y\to X$ is a covering for the fpqc topology, there is an $S$-algebra $T$ which is faithfully flat (= pure + flat) over $R$. Let $m\in M$ be such that $b(m\cdot m)\notin J^2$, and let $t$ denote the image of $(0,m)\in S$ in $T$.  
 
 The exact sequence $0 \to I \to R \to R/I\to 0$ gives rise to an exact sequence $0 \to IT = I\otimes_R T \to T \to T/IT \to 0$ since $T$ is flat over $R$. But $I\otimes (0,m)$ maps to $0$ in $S$. Therefore $It= 0$. 
 
 On the other hand, $t^2$ is the image of $b(m\cdot m)$ in $T$, which is not in $J^2$. Since $R\to T$ is pure, $J^2 T \cap R = J^2$, so that $t^2\notin J^2T$. A fortiori, $t\notin JT$. Since $I$ is a finite $R$-module and $T$ is flat over $R$, $JT = {\rm{Ann}}_T\, IT$.  This implies $It\neq 0$, a contradiction. 
   \qed
 
 \begin{example}\label{ex 4.2}  $R=\, \mathbb Z/4 \mathbb Z, \; M = \mathbb Z/2 \mathbb Z $. Let $b$ be given by multiplication in $\mathbb Z/2$ followed by the embedding of $\mathbb Z/2$ in $ \mathbb Z/4 $ given by the action on the element $[2]\in  \mathbb Z/4 $. It satisfies our assumptions ($J = (2), J^2 = 0, \,q(1)= 2 \neq 0 $ in $\mathbb Z/4$), and we get a non reduced example in dimension $0$ of a canonical finite covering $Y\to X$ which is not a covering for the fpqc topology.  \end{example}
 
  \begin{example}\label{ex 4.3} $R=\, k[x,y]/(xy), \; M = k[y]={R\over (x)} $.
We denote by $\bar{r}$ the image of $r$ in $M$. 
Let $b$ be given by $b(m\cdot n)= ymn\in yk[y] \subset R$. 
  It satisfies our assumptions ($J = (y), \,q(1)= y \notin J^2$), and we get a reduced non normal example in dimension $1$ of a canonical finite covering $Y\to X$ which is not a covering for the fpqc topology.    
  
     In this example, not only $X$ but also $Y$ is reduced: indeed, $(r,m)^2 =0 $ implies 
$r^2 + ym^2= 0$.
Hence ${\bar{r}}^2+ym^2=0$  implies $\bar r=m=0$
since $y$ is not a square in $k(y)$: a contradiction, unless $r=m=0$.  \end{example} 
    
    In both examples, $b$ is actually nonsingular, i.e. identifies $S$ with its $R$-dual, as in the context of Frobenius algebras.

\subsection{} Let us turn to examples where $X$ is {\it normal}. 
We shall rely on a construction due to Raynaud and Gruson \cite[1.4.1.1]{RG}, who proved
for the first time that  {\it the fpqc topology is strictly coarser than the canonical topology}\footnote{the authors are not aware of any other such example in the meantime, before the above more elementary examples (of Frobenius type).}.
However they did not provide an example of a \emph{finite} canonical covering which is not an fpqc covering. 
In the sequel we will refine their construction and provide such an example.

  \begin{example}\label{ex 4.4}  Let us discuss the simplest instance in the Raynaud-Gruson class of examples \cite[1.4.1.1]{RG}:  $X $ is the completed germ (at the vertex) of the standard quadric cone 
    over a field $k$ of characteristic $\neq 2$, i.e. $\Spec\, k[[u,v,w]]/(w^2-v^2-u^2)$. This is a normal quotient singularity (in particular geometrically unibranch). 
   The natural double covering $ \hat{\mathbb A}^2 = \Spec \, k[[x,y]] \to X= 
    \bigslant{\hat{\mathbb A}^2}{(\mathbb Z/2\mathbb Z)}$ (where $\mathbb Z/2\mathbb Z$ acts by $-1$ on both coordinates $x,y$, and $u=2xy, v=x^2 -y^2, w=x^2+y^2 $) produces a non trivial $2$-torsion line bundle $L$ on $X^\ast = X \setminus \{0\}$.

    \bigskip

    Let $f': Y'\to X$ be the affinization of the total space of the corresponding $\mathbb G_m$-torsor on $X^\ast$: $\,Y' = \Spec (\oplus_{n\in \mathbb Z} \Gamma(X^\ast, L^n))$. 
This is a morphism of finite type, which is equidimensional   \cite[1.4.1.1]{RG}.
This fact combined with the fact that $X$ is geometrically unibranch implies 
 the existence of a  
morphism $Y\to Y'$ 
such that the composition $f:Y \to Y'\stackrel{f'}\to X$ is quasi-finite and 
dominant \cite[IV.14.4.4 \& 14.5.9]{EGA}\footnote{these references also show that $Y\to Y'$ can be chosen to be a
locally closed immersion, that is (in a terminology which makes cameo
appearances in EGA): $Y $ is a quasi-section of $f'$. Amusingly, the
notion of quasi-section is nowhere formally defined in EGA: its first
occurrence is \cite[p. 200]{EGA}, where only the meaning of the expression
``to have enough quasi-sections" is explained (with a typo: ``closed
subscheme" should be replaced by ``locally closed subscheme").}.
 Since $X$ is the spectrum of a complete local domain,
$Y$ is a sum $Y_1\amalg Y_2$ where $Y_1$ is a finite covering of $X$ \cite[II.6.2.5]{EGA}.
Replacing $Y$ by $Y_1,$ we can assume that $Y\stackrel{f}\to X$ is a finite covering.
 
 \smallskip

Let us prove that $f$ is not an fpqc covering. 
The $\mathbb G_m$-torsor $Y'$  becomes trivial after pull-back along any morphism $Z^\ast \to X^\ast$ such that $Z^\ast\times_{X}  Y'\to Z^\ast$ has a section.
Assume by contradiction that there is an affine $Y$-scheme $Z$ which is faithfully flat over $X$, and let $Z^\ast $ be the preimage of $Z$ over $X^\ast$. Then $Z^\ast\times_{X}  Y'\to Z^\ast$ has a section, hence the inverse image of $L$ on $Z^\ast$ is trivial, and so is its extension to $Z$. Applying flat base change along $Z\to X$ to the inclusion $X^\ast \to X$, and then faithfully flat descent,  one concludes that $L$ itself is trivial, a contradiction.
Therefore $Y\to X$ is not a covering for the fpqc topology.
  
 On the other hand,
 $Y\stackrel{f}{\to} X$ is a canonical finite covering; in fact, any finite covering of  $X$  is canonical: indeed,
 the quotient map $\mathbb A^2_k\to X$ is a canonical covering (cf. \ref{ex 4.7} below), and 
 any finite covering of
 $\mathbb A^2_k$ is canonical
      since any normal finite covering of $\mathbb A^2_k$ is faithfully flat.  \qed

      \end{example} 
  
    \begin{example}\label{ex 4.7} This example suggests the following question: is the natural double covering of the quadric cone $  {\mathbb A}^2 \to X= \bigslant{ {\mathbb A}^2}{(\mathbb Z/2\mathbb Z)}$  an example of a canonical covering which is not a covering for the fpqc topology?    
  
  It is! 
Since $ k[x,y]^{\mathbb Z/2} \to k[x,y]$ has a natural retraction (the Reynolds operator),  $  {\mathbb A}^2 \to X$ is a canonical covering.
    On the other hand, assume that it is a covering for the fpqc topology. We may localize and complete at the vertex. Let $Y\to X$ be the finite covering of Example \ref{ex 4.4}, which is not a covering for the fpqc topology. But $Y\times_X \hat{\mathbb A}^2$ is a covering for the fpqc topology since any normal finite covering of $\hat{\mathbb A}^2$ is faithfully flat. Therefore $Y\times_X \hat{\mathbb A}^2\to \hat{\mathbb A}^2\to X$ is a composition of coverings for the fpqc topology, which factors through $Y\to X$: a contradiction. 
  
  In a sense, this example looks simpler and more natural than Examples \ref{ex 4.2} and  \ref{ex 4.3}. However, we do not know an elementary proof that it is actually an example; two different proofs  (no more elementary) are given in Theorem~\ref{T7} below. 
 \end{example}

\subsection{} However, there is no hope to find an example  with a {\it regular} base $X$: 

\begin{theorem}\label{T2}\cite{A1}   Any finite covering of a regular Noetherian scheme is a covering for the fpqc topology.\qed
\end{theorem}

\begin{remark} To emphasize the contrast between the fpqc topology and the fpqc pretopology, let us remind that {\it a finite covering $Y\to X$ of a regular Noetherian scheme $X$ is flat if and only if $Y$ is Cohen-Macaulay}. In particular, any Noether normalization of a non Cohen-Macaulay integral affine variety (e.g. $\Spec \, k[x^4,x^3y,xy^3,y^4]\to \Spec \, k[x^4, y^4]=\mathbb A^2_k$) gives an example of a finite covering for the fpqc topology which is not an fpqc covering in the usual sense, i.e. for the pretopology.
\end{remark}

Clearly, Theorem \ref{T2} is stronger than Theorem \ref{T1}.
It is actually much stronger: in case the base ring $k$ is a field of characteristic $0$, Theorem \ref{T1} is essentially trivial, while the only available proofs of Theorem \ref{T2} use reduction modulo $p$ for infinitely many $p$'s, an ultraproduct-type argument, and the proof in characteristic $p$, which is highly non trivial.  

Theorem~\ref{T2} reduces immediately to the affine situation, where it is stated as 0.7.2 in \cite{A1}. It is actually {\it equivalent} to Hochster's conjecture on the {\it existence of big Cohen-Macaulay algebras} for complete local rings, proved by Hochster if $k$ is a field (cf. \cite{H2}), and by the first author in general.  
 Theorem \ref{T2} unveils the {\it geometric meaning} of this (ex-)conjecture.
 
 Let us briefly comment. Given a Noetherian local ring $S$, a (not necessarily finitely generated) $S$-algebra $T$ is said to be a (big) Cohen-Macaulay $S$-algebra if every secant sequence  in $S$  (see footnote~\ref{F1}) becomes regular in $T$\footnote{to insist on the condition that \emph{all} secant sequences in $S$ become regular in $T$, some authors prefer to use the terminology \emph{balanced} big  Cohen-Macaulay algebra. }. If $S$ is regular, this amounts to requiring that $T$ is faithfully flat over $S$.
    In particular, if $S$ is local complete and $p$ is not a zero-divisor in mixed characteristic $(0,p)$, it is a finite extension of a complete regular ring $R$ (by Cohen's structure theorem), and up to substituting $R$ by $S$, we see that an $S$-algebra is Cohen-Macaulay if and only if it is faithfully flat over $R$. This explains why Theorem~\ref{T2} implies Hochster's conjecture. The slightly more technical opposite implication  is proven in \cite[4.3]{A1}.  
  
  \begin{remark}\label{r 5.8} Theorem \ref{T2} also holds when $f$ is an {\it integral {surjective} morphism}. Unlike the case of Theorem \ref{T1}, this is not a trivial extension of \ref{T2}, but this requires an enhancement of the existence of Cohen-Macaulay $R$-algebras. Along the argument of \cite[0.7.2]{A1}, this follows from the existence of {\it absolutely integrally closed} big Cohen-Macaulay algebras, which exist in char. $p$ by \cite{D}, in char. $0$ by \cite{DRG}, in mixed characteristic by \cite{BS} (actually, this already follows from the existence of Cohen-Macaulay $R$-algebras which are algebras {over the absolute integral closure of $R$}, which is easier to establish).
      \end{remark}
 
\begin{remark}\label{q 5.9} Theorem \ref{T2} leaves open the following question.
 {\it Is any canonical covering (say, of finite type) of a regular base a covering for the fpqc topology?}
    A special case reads (cf. Remark \ref{qF}): given a regular ring $R$, a (finitely generated) $R$-algebra $S$ and an $R$-faithfully flat $S$-module, does there exist an $R$-faithfully flat $S$-algebra?   
    \end{remark}
       
   \subsection{Frobenius.} Let $k$ be a field of characteristic $p$ and $X$ be an affine {Noetherian} $k$-scheme. As we have mentioned before, there exist non regular examples for which $F_X$ is a covering for the canonical topology (and often a finite covering).  What about the fpqc topology? {The following result {is a slight extension of} Kunz' theorem \cite{Ku}.
   
   \begin{theorem}\label{T6} $X$ is regular if and only if $F_X$ is a covering map for the fpqc topology.
   \end{theorem} 
   
 \begin{proof} By Kunz' theorem (see also \cite{BS1} for a new proof), $X$ is regular if and only if $F_X$ is (faithfully) flat. Therefore, it remains to prove that if a composed map $h: Z \stackrel{g}\to X \stackrel{F_X}\to X $ is faithfully flat, then $g$ is a canonical covering: indeed, this implies that $F_X$ is faithfully flat (Lemma \ref{L1}). 
 Let us build the diagram
 $$\xymatrixcolsep{3em}
\xymatrixrowsep{1.5em}
\xymatrix{& X \ar[d]^{F_X} \\
 Z \ar[ur]^g \ar[d]^{F_Z}  \ar[r]^{h} & X \ar[d]^{F_X}  \\
   Z \ar[ur]^g \ar[r]^{h}  &X   .\\
}
$$
Since  $gF_Z= F_Xg= h $, this is a commutative diagram, and since $h$ is faithfully flat, hence a canonical covering, so is $g$ as required. 
  \end{proof} }

\section{Finite coverings which are coverings for the fpqc topology but not for the fppf topology}

In view of Theorem \ref{T2}, it suffices to produce a regular ring $R$ and a finite extension $R\to S$  such that there is no finitely generated $S$-algebra $T$ which is faithfully flat over $R$.

\begin{example}\label{ex 5.1} Following Bhatt \cite{B2}, let $(A, L)$ be a  polarized abelian variety
of dimension $d>1$ \emph{over a field $k$ of characteristic $p>0$}. Let 
$S$ be the completion at the origin of the ring of functions of
the cone over  $(A, L)$: $S = \hat\oplus_{n\in \mathbb N} \Gamma (A, L^n)$. This is a complete local domain of dimension $d+1$, which can thus be written as a finite extension of a regular complete local domain $R$ (Cohen). According to \cite{B2}, no finite extension of $S$ is Cohen-Macaulay (equivalently, faithfully flat over $R$). 

\smallskip

 Let us assume, by contradiction, that there exists a finitely generated $S$-algebra $T$ which is faithfully flat over $R$.
 This implies the existence of a morphism $\Spec R'\to \Spec T$ 
such that the composition $\Spec R'\to \Spec T\to \Spec R$ is quasi-finite and faithfully flat \cite[IV. 17.16.2]{EGA}.
 Applying \cite[II. 6.2.5]{EGA} as in Example~\ref{ex 4.4}, we may assume 
 that 
 $R'$  is even finite over $R$.
We regard $R'$ as an $S$-algebra via the composition $S \to T \to R'$. 
This is a finite extension (indeed,  in the composition $R\to S\to R'$, both $S$ and $R'$ are finite $R$-algebras of the same dimension and $S$ is a domain). 
By Bhatt's result, $R'$ cannot be faithfully flat over  $R$, a contradiction.
 \qed
\end{example}

\begin{example} A similar, but simpler example, is given by any normal complete local domain $S$  \emph{over a field $k$ of characteristic $0$}, which is not Cohen-Macaulay. Indeed, any finite extension of $S$ is pure, and by Lemma \ref{l4.4}, no such extension is Cohen-Macaulay. One concludes as in the previous example (writing $S$ as a finite extension of a regular complete local domain $R$). 
\end{example}

 \begin{remark} Given a finite extension $R\to S$ of complete local domains, with $R$ regular, it is not known whether there exists an $R$-free finitely generated $S$-module, but there is a countably generated such module $M$ \cite{Gri}\footnote{the result is stated there under the assumption that $k$ is a field, but this assumption is only used through the existence of big Cohen-Macaulay $S$-algebras, which was available only under this assumption at the time.}.  It would be interesting to know whether there is a {\it countably generated $S$-algebra} which is $R$-free.  \end{remark}

\section{On ``weak functoriality" of coverings for the fpqc topology}

\subsection{} We consider a commutative square in ${\bf{Aff}}_k$  
{\small\begin{equation}\label{CD1}  
\xymatrix   @-1.2pc{     Y'  \ar[d]_{f'}  \,  \ar@{->}[r]     & Y   \ar[d]^f    \\ X'  \,  \ar@{->}[r]^h  & X }  \end{equation}} 
where the vertical maps are {\it covering maps for the fpqc topology}. This means that \eqref{CD1}  
 can be completed into a diagram
{\small\begin{equation}\label{CD2}  \xymatrix   @-1.2pc{     Z'  \ar[d]_{g'}   \,     & Z  \ar[d]^g      
  \\  \;\, Y'  \ar[d]_{f'}   \,  \ar@{->}[r]   & \;\, Y  \ar[d]^f   
      \\ X'  \,  \ar@{->}[r]    & X   }  \end{equation} }
      where the compositions of vertical maps are {\it faithfully flat}. 
  
      We say that the diagram \eqref{CD1} satisfies ``fpqc weak functoriality" if one can choose \eqref{CD2} in such a way that it extends to a commutative diagram\footnote{this condition is trivially satisfied if the bottom square is cartesian.} 
{\small\begin{equation}\label{CD3}\xymatrix   @-1.2pc{     Z'  \ar[d]_{g'}   \,  \ar@{->}[r]    & Z  \ar[d]^g      
  \\  \;\, Y'  \ar[d]_{f'}   \,  \ar@{->}[r]   & \;\, Y  \ar[d]^f  
      \\ X'  \,  \ar@{->}[r]    & X  . }   \end{equation} }
      
This may seem a futile request at first sight, but this turns out to be closely related to a key property for the whole skein of the so-called homological conjectures in Commutative Algebra: weak functoriality of Cohen-Macaulay algebras (conjectured by Hochster, and proved by him and Huneke when $k$ is a field \cite{H2}\cite{HH2}, and by the first author in general \cite{A2} using perfectoids), which asserts that any morphism $Y'\to Y$ which corresponds to a local homomorphism of complete local domains sits in a commutative square in ${\bf{Aff}}_k$ 
{\small \begin{equation}\label{CD4}  \xymatrix   @-1.2pc{     Z'  \ar[d]_{g'}  \,  \ar@{->}[r]     & Z   \ar[d]^g   \\ Y'  \,  \ar@{->}[r]  & Y }  \end{equation} }where $\cO(Z) $ and $\cO(Z')$ are Cohen-Macaulay algebras for $\cO(Y) $ and $\cO(Y')$ respectively. 
          
 \subsection{} Fpqc weak functoriality holds in the situation of Theorem \ref{T2}: 
 
\begin{theorem}\label{T3}  If $X,X'$ are regular and $f,f'$ are finite coverings, then \eqref{CD1} satisfies ``fpqc weak functoriality".  
 \end{theorem}  
  We shall see {a geometric consequence of this result in Part III, Theorem \ref{T5}.}

 \begin{proof} $(1)$ In the case where $Y'\to Y$ corresponds to a local homomorphism of complete local domains, this follows from superposing a square like \eqref{CD4} on top of the square \eqref{CD1}; the Cohen-Macaulay property translates into the faithful flatness of $fg$ and $f'g'$ (relying upon \cite{H2},\cite{HH2},\cite{A2}, as said before).
 
 $(2)$ Let us reduce the general case to $(1)$. For any $x\in X$, the normalization $Y_{[x]} := (Y\times_X \Spec\, \hat \cO_{X,x})_{nor}$ is a finite disjoint sum of spectra of complete local domains. Same definition and property for $Y'_{[x']}$. By $(1)$, for any $x'\in X'$, there are commutative diagrams 
 \begin{equation}\label{CD5} \small{  \xymatrix   @-1.2pc{     Z'_{x'}  \ar[d]    \,  \ar@{->}[r]    & Z_{x'}  \ar[d]       
  \\  \;\, Y'_{[x']}  \ar[d]    \,  \ar@{->}[r]   & \;\, Y_{[h(x')]}  \ar[d] 
      \\  \Spec\, \hat \cO_{X',x'}  \,  \ar@{->}[r]    &  \Spec\, \hat \cO_{X,h(x')}  }}   \end{equation} 
where the composed vertical maps are faithfully flat. On the other hand, for any $x\notin h(X')$, there are faithfully flat compositions $Z_x\to Y_{[x]} \to \Spec\, \hat \cO_{X,x}$. Taking coproducts in ${\bf{Aff}}_k$ (i.e. spectra of products of rings, {not disjoint unions of spectra}), we obtain the desired diagram \eqref{CD3}, taking into account the fact that since $X$ and $X'$ are Noetherian, a (possibly infinite) coproduct of flat maps with target $X$ or $X'$ is flat.
\end{proof}
{  \begin{remark}\label{Rwf7} Fpqc weak functoriality is related to the following open question. {\it Given a directed inverse system $(Y_\lambda)$ of coverings of $X$ for the fpqc topology, is the limit $\lim Y_\lambda$ still a covering of $X$ for the fpqc topology?} If $X$ is regular and the $Y_\lambda$ are finite coverings, this holds true (once translated in terms of Cohen-Macaulay algebras, this follows from \cite[3.2]{D}). Note also that this is true if fpqc is replaced by the canonical topology.
  \end{remark} }

 \section*{\\ Part III. The finite topology on affine schemes. 
  Splinters and their fpqc analogs}  
 
We fix a Noetherian base ring $k$.

 \section{The finite and qfh topologies} 
 
  The \emph{finite topology} on ${\bf{Aff}}_k$ is the Grothendieck topology generated by 
 finite coverings.

 A morphism $ Y\stackrel{f}{\to} X$ in ${\bf{Aff}}_k$ is a covering for the finite topology if and only if there is a finite covering  $X'\to X $ such that the pull-back {$Y':=Y\times_X X'\to X'$} has a section. 
 
 The finite topology is the natural context for the problems discussed in this paper.  One may wonder why this elementary-looking topology did not already appear long ago, but only 
recently \cite{GK} (in the context of $k$-schemes of finite type).  One reason might be that, until Voevodsky's introduction of new topologies in the theory of motives, the consensus was that a ``well-behaved" Grothendieck topology on schemes had to be subcanonical. After Voevodsky's successful introduction of many non subcanonical topologies and their massive use by Ayoub and others, the idea emerged that a ``well-behaved"  topology might rather be one for which {\it one can compute local rings}, or conservative families of points, in order to be able to check concretely that morphisms of sheaves are isomorphisms. In this respect, several classical subcanonical topologies do not fulfill the latter criterium (fppf?, fpqc, canonical), whereas several non subcanonical topologies do, among them the finite topology \cite{GK}. In \cite{GK}, it is proven that {\it the local rings for the finite topology are the absolutely integrally closed domains}. 
 
 The topology generated by finite coverings and Zariski open coverings coincides with Voevodsky's qfh topology (extended by Rydh to the non Noetherian setting) \cite[8.4]{Ry}. The local rings for this topology, in the context of affine $k$-schemes of finite type, are the absolutely integrally closed {\it local} domains \cite{GK}.
Coverings for the qfh topology include fppf morphisms, since they admit flat quasi-sections (\cite[IV. 17.16.2]{EGA}), hence the qfh topology is finer than the fppf topology. In fact, the fppf topology is generated by Nisnevich coverings and finite flat coverings \cite[p.8]{GK} (see also \cite{Sch}).

This suggests the following naive question: is the canonical topology generated by the fpqc topology and by finite canonical coverings?  

 A weaker and more plausible question is: 
\begin{question}\label{q2} 
{{\it Is the canonical topology on ${\bf{Aff}}_k$ generated by the  fpqc topology and by  morphisms $Y\to X$ such that $\cO_X\to \cO_Y$ 
splits as $\cO_X$-modules? }}
\end{question}

This is at least true if the base is Noetherian; more precisely:

\begin{proposition}\label{rq2} Let $Y\stackrel{f}\to X$ be a canonical covering of a Noetherian scheme $X$ in ${\bf{Aff}}_k$. One can construct a commutative square in ${\bf{Aff}}_k$ 
$$\small\xymatrixcolsep{3em}
\xymatrixrowsep{1.5em}
\xymatrix{
Y'\ar[r]^{f'} \ar[d]& X'  \ar[d]^e\\
Y\ar[r]^{f} & X\\ }
$$ 
such that $e$ is faithfully flat while $\cO(X')\stackrel{f'^\ast}\to \cO(Y')$ 
 splits in $\cO(X')\lMod$. This construction is functorial in $f$. 
\end{proposition}
\begin{proof} We set $X'= \Spec\, (\prod\limits_{ {x}\in X} {\hat{\cO}}_{X,x})$. Because $X$ is Noetherian, the codiagonal map $X'\stackrel{e}\to X$ is faithfully flat. We set $Y'= \Spec\, (\prod\limits_{ {x}\in X} ({\cO(Y)\otimes_{\cO(X)}\hat{\cO}}_{X,x}))$. By Proposition~\ref{Prop:2+}, the natural map
$f'^\ast := \prod \alpha_{ {\hat{\cO}}_{X,x}}:\,  \cO(X') \to \cO(Y')$ splits in $R'\lMod$.

This construction is clearly functorial: if $$\small\xymatrixcolsep{3em}
\xymatrixrowsep{1.5em}
\xymatrix{
Y_1\ar[r]^{f_1}\ar[d]  & X_1 \ar[d] \\
Y_2\ar[r]^{f_2} & X_2\\ }
$$ is a commutative square in ${\bf{Aff}}_k$, $\cO({X'_2})\to \cO({X'_1})$ (resp. $\cO({Y'_2})\to \cO({Y'_1})$) is built by sending the component of the product indexed by $x_2\in X_2$ (resp. $y_2\in Y_2$) to the component indexed by its image in $X_1$ (resp. $Y_1$). 
 \end{proof}

{We close this section with the following elementary lemma, which is however of little use:

\begin{lemma} Any finite covering $Y\to X$ in ${\bf{Aff}}_k$ is a composition $Y\to W\to X,$ where $Y\to W$ is a closed immersion and $W\to X$ is a finite flat covering.
\end{lemma} 
 
Indeed, if $\underline s = (s_i)$ is a finite family of generators the $\cO(X)$-algebra $\cO(Y)$, and $p_i\in \cO(X)[x_i]$ are monic polynomials such that $p_i(s_i)=0$, one can take $W =\Spec\, \cO(X)[\underline x]/(\underline p)$.}
\qed

   \section{When every finite covering is canonical: splinters}
   \subsection{} Comparing the finite and canonical topologies suggests the following notion\footnote{the notion/name of ``splinter" goes back to \cite{FMa} (Theorem \ref{T1} was already known when $k$ is a field). It is closely related to the notion of ideally integrally closed domains studied in \cite[\S 2]{H1}.}:

    \begin{definition} {{ An affine Noetherian $k$-scheme $X= \Spec \, R$ is called a \emph{splinter} if every finite covering (equivalently, any covering for the finite topology) is a canonical covering. }}  \end{definition}  By Proposition \ref{P:2}, this means that {\it every faithful finite $R$-algebra splits as an $R$-module} (one can even replace ``finite" by ``integral" because filtered colimits of pure ring maps are pure). 
   
According to Theorem \ref{T1}, any {\it regular} $X$ is a splinter. One may reinterpret the theory of splinters as the discussion of the (non-)optimality of this result: namely, beyond regular schemes, which (affine Noetherian) schemes are splinters? 

We first record the following easy and well-known facts:

\begin{proposition} \begin{enumerate} 
\item The notion of splinter is Zariski-local.
 \item Splinters are normal.
\item The converse holds in characteristic $0$. 
 \end{enumerate}\end{proposition}
For $(1)(2)$, see for instance \cite[2.1.3, 2.1.2]{DT}. For $(3)$, see the discussion {after Theorem \ref{T1}. Example \ref{ex 3.4} shows that the converse no longer holds in positive or mixed characteristic. 
 It is known that an excellent affine local (Noetherian) $k$-scheme $ \Spec \, R$ is a splinter if and only if so is its completion $ \Spec \, \hat R$ \cite[Th. C]{DT}.} \qed

\subsection{}  What are splinters in characteristic $p$? 
 There is an extended literature on the subject, which belongs to the field of ``$F$-singularities".
  
 \smallskip A conjecture predicts that {\it splinter = $F$-regular}, a kind of characteristic $p$ analog of log-terminal singularities\footnote{sometimes, $F$-regular is replaced by strongly $F$-regular, which means that for any nonzero $a\in R$ there is a power $F^e$ of $F$ such that the composition of $F^e$ followed by multiplication by $a$ splits, cf. \cite[2.3]{LMa}; these notions coincide in the $\mathbb Q$-Gorenstein case \cite[3.5]{TW} On the analogy between $F$-regular and log-terminal singularities, see \cite{HaW}.  In mixed characteristic, using perfectoid big Cohen-Macaulay algebras \cite{MaS}, there is an analog of $F$-regularity thanks to Ma and Schwede's work. }. This conjecture has been solved in the $\mathbb Q$-Gorenstein situation \cite{Si}.
   
  { \begin{example}\label{ex 9.3} The spectra of a large class of cluster algebras  
  - the so-called locally acyclic cluster algebras - are splinters \cite{BMRS}. \end{example} }
  
 \begin{theorem}\label{T4} 
 In characteristic $p$, excellent splinters are Cohen-Macaulay (Hochster and Huneke)\footnote{in fact, one can replace excellent by locally excellent, i.e. such that the localization of the ring at any maximal ideal is excellent.}.
 The same is true in mixed characteristic (Bhatt). 
 \end{theorem}
 
\begin{proof} In characteristic $p$, this is a straightforward consequence of the theorem of Hochster and Huneke: for any excellent local domain $R$ of characteristic $p$, its absolute integral closure $R^+$ (i.e. the integral closure of $R$ in an algebraic closure of its fraction field) is a Cohen-Macaulay $R$-algebra  \cite{HH} (see also \cite{HL})\footnote{this no longer holds in characteristic $0$ if $\dim R \geq 3$; nor in mixed characteristic $(0,p)$ (by localizing at a prime not containing $p$)  if $\dim R \geq 4$.}.
  Indeed, since $R^+$ is a filtered colimit of finite extensions of the splinter $R$, it is pure over $R$. It follows that $R$ itself is a Cohen-Macaulay ring \cite[2.1.g]{HH2}. 
  
\smallskip

In mixed characteristic,   Bhatt recently proved the analog of the theorem of Hochster and Huneke  (\cite[Th. 1.1]{B4}) and the second assertion follows in a similar way (\cite[Cor. 5.9]{B4}). Bhatt's proof relies on prismatic techniques and on a new Riemann-Hilbert correspondence in mixed characteristic.
  \end{proof}

      \subsection{}\label{subs 7.3} Let us come back to the general situation and start with a simple observation: if $\tau'$ and $\tau$ are Grothendieck topologies on a category which admits fiber products, the condition that any $\tau'$-covering of an object $X$ is a $\tau$-covering can be tested after replacing $X$ by any $\tau$-covering. In particular (for $\tau'=$ finite, $\tau=$ canonical), {\it $X$ is a splinter if for some canonical covering $X'\to X$, $X'$ is a splinter.}

   \begin{example}\label{ex 7.1}  The quadric cone $X$ over a field of characteristic $\neq 2$ is a splinter by the argument explained in Example~\ref{ex 4.7}: the quotient map $\mathbb A^2_k\to X$ is a canonical covering, and $\mathbb A^2_k$ itself is a splinter,
     since any normal finite covering of $\mathbb A^2_k$ is faithfully flat. 
 \end{example} 

The following result summarizes several previous works on canonical coverings by a regular scheme. 
 
\begin{theorem}\label{T5} Let $X$ be an affine Noetherian $k$-scheme. Assume that 

$\,\;(0)$ there exists a canonical covering $Y\stackrel{f}{\to} X$  in ${\bf{Aff}}_k$ with $Y$ regular. 

Then:
\begin{enumerate} 
\item  $X$ is a splinter, hence normal;
\item $X$ is Cohen-Macaulay;
\item (Assuming $X$ excellent) $X$ has  rational singularities.
\end{enumerate}
\end{theorem} 

\begin{proof}  $(1)$ is a straightforward consequence of Theorem~\ref{T1} combined with the previous simple observation \ref{subs 7.3}. 

 \smallskip\noindent $(2)$: after reducing to the case where $\cO(X)$ is a complete local domain and applying Cohen's theorem if necessary, one may assume that $X$ is a finite covering of a regular $W$, and one has to show that $X\stackrel{h}\to W$ is flat. By fpqc weak functoriality (Theorem \ref{T3}), there is a commutative diagram 
 \begin{equation}\label{CD9}\xymatrix   @-1.2pc{     Z'  \ar[d]_{g'}   \,  \ar@{->}[r]^{f'}    & Z  \ar[d]^g      
  \\  \;\, Y  \ar[d]_{id_Y}   \,  \ar@{->}[r]^f   & \;\, X \ar[d]^{h} 
      \\ Y \,  \ar@{->}[r]    & W  . }   \end{equation} such that $g'$ and $hg$ are faithfully flat. Since $f$ and $g'$ are canonical coverings, so is  $fg'=gf'$, hence also $g$. Since $hg$ is flat, it follows that $h$ is flat (Lemma \ref{L1})\footnote{before \cite{A2}, Heitmann and Ma got just ``enough" weak functoriality to work out this argument in the mixed characteristic case. Notice that the weak ``weak functoriality" statement in \cite[4.4]{A1} would already be enough to settle the issue, except for an assumption of separability of the extension of residue fields.}. 
  
\smallskip\noindent $(3)$: after reduction to the local case via Lemma~\ref{L2},  it is proved that $X$ has ``pseudo-rational" singularities, respectively in \cite{Sc} in characteristic $0$ (extending Boutot's work \cite{Bo}\footnote{Boutot used Grauert-Riemenschneider's vanishing theorem, assuming that $f$ is essentially of finite type over $k$. Schoutens combines Smith's techniques in characteristic $p$ with ultraproducts.}), in \cite{Sm} in positive characteristic, and in \cite{HeMa} in mixed characteristic (using perfectoids); finally, it is proved in \cite{Ko} that ``pseudo-rational" singularities is equivalent to rational singularities. \end{proof} 
 
 Let us briefly comment. The classical notion of rational singularity postulates  a resolution of singularities $\tilde X {\to} X\,$: it prescribes that $H^i(\tilde X, \cO_{\tilde X})=0$ for $i>0$. Recently,  Kov\'acs showed that this property extends to any ``Macaulayfication" $\tilde X$. He redefined in this way the notion of rational singularities in arbitrary characteristic and proved that it is equivalent to the earlier (a priori weaker) notion of pseudo-rational singularity introduced by Lipman and Teissier \cite{Ko}.

 \smallskip   {\footnotesize{ \begin{remark} There is a derived version of the notion of splinter: {\it an affine Noetherian $k$-scheme $X $ is called a \emph{derived splinter} if for any proper surjective morphism $ Y \stackrel{f}{\to} X$, $\cO_X \to Rf_\ast \cO_Y$ splits in $D^b({\cO_X})$.}
  
 This notion was introduced by Bhatt in light of De Jong's derived version of the direct summand conjecture which, in these terms, predicted that any regular $X$ is a derived splinter.  Refining the perfectoid techniques of  \cite{A1}, Bhatt proved this {\it derived direct summand conjecture} \cite{B3}. 
 
 According to \cite{E}, $\cO_X \to Rf_\ast \cO_Y$ splits if and only if $f$ is an effective derived descent morphism. In this perspective, an equivalent form of Bhatt's theorem reads: any proper surjective morphism  with target a regular affine scheme is an effective derived descent morphism.
  
 Any derived splinter is a splinter. In \cite{B1}, resp. \cite[Rem. 1.7]{B4}, Bhatt proved the converse in positive characteristic, resp. in mixed characteristic using techniques from \cite{BS}. In characteristic $0$, this fails: splinter = normal, while derived splinters (essentially of finite type over $k$) correspond to rational singularities \cite{Ko}\cite{B1}.
 
 As an application, let us cite Ma and Schwede's characterization of regularity of $X$ by way of the finite projective dimension of $Rf_\ast \cO_Y$ for a regular alteration $ Y \stackrel{f}{\to} X$ \cite{MaS2}.
 \end{remark} }}

\section{Fpqc analogs of splinters?}  
  \subsection{}  The problematics of splinters deals with the relationship between the finite topology and the canonical topology on affine schemes. What if one replaces the canonical topology by the fpqc topology? 

{ \begin{question}\label{q3} {{\it Which affine Noetherian schemes have the property that every finite covering is a covering for the fpqc topology?}} 
(Equivalently: which Noetherian rings $R$ have the property that for every finite extension $S$, there is an $S$-algebra faithfully flat over $R$?)\end{question}}

 \smallskip Let us call such objects ``fpqc analogs of splinters"\footnote{one might wonder whether fpqc analog of splinters also have the property that integral coverings are fpqc coverings. This is true for regular schemes as explained in Remark~\ref{r 5.8}. We do not know the answer in general; a filtered colimit argument does not work formally as explained in Remark~\ref{Rwf7}. }. Theorem \ref{T2} says that regular schemes have this property. 
 {It turns out that in positive characteristic, this theorem is sharp: 
 
 \begin{proposition} In positive characteristic, the only $F$-finite ``fpqc analogs of splinters" are regular rings. 
 \end{proposition}   
 
 \begin{proof} This is an immediate consequence of Theorem \ref{T6}. 
 \end{proof} }
  
    {\footnotesize{ \begin{remark} The notion of ``fppf analog of a splinter" is of very limited scope. 
    For instance, the spectrum $X = \Spec\, R\,$ of a regular complete local domain of dimension $d$ cannot be such an analog if $d \geq 3$ in equal characteristic (resp. if $d\geq 4$ in mixed characteristic). 
    Namely, assume by contradiction that for every finite extension $S$ of $R$, there is an $S$-algebra $T_S$ which is faithfully flat of finite presentation over $R$.  Arguing as in Example~\ref{ex 5.1}, one may assume that $T_S$ is a domain and a finite extension of $R$.    
  This implies that $R^+ =  {\rm{colim}}\, T_S$ is faithfully flat over $R$, hence a Cohen-Macaulay $R$-algebra. 
This is wrong in characteristic $0$ if $d \geq 3$, and in mixed characteristic if $d\geq 4$. In positive characteristic, instead, the result follows from Example  \ref{ex 5.1}. 
      \end{remark}  }}
 
   \subsection{}\label{reg} In search of potential examples of {(affine Noetherian)} non regular ``fpqc analogs of splinters", one might look for a covering $Y\stackrel{f}{\to} X$  for the fpqc topology, with $Y$ regular and $X$ non regular (such an $f$ cannot be flat, by 
   \cite[IV. 6.5.2]{EGA}). Following the observation at the beginning of subsection \ref{subs 7.3} (for $\tau' =$ finite, $\tau =$ fpqc),  Theorem \ref{T2} applied to the fibered product of $f$ with an arbitrary finite covering of $X$ would show that $X$ is such an example. 
   
     \subsection{} This suggests a basic question about the fpqc topology.
Many properties of schemes and morphisms of schemes descend with respect to faithfully flat morphisms (i.e. with respect to the fpqc {\it pretopology}); this is one of Grothendieck's great insights (\cite[IV. 2.2]{EGA}). 
As for morphisms, the properties under consideration are stable by base change; therefore, the questions whether they descend with respect to the fpqc {\it pretopology} or with respect to the fpqc {\it topology} are equivalent. 

As for schemes, this is quite another story. In \cite[IV. 6]{EGA}, there is a long list of properties of schemes which descend with respect to the fpqc {\it pretopology}. Whence the question:
{\it which properties of  
schemes descend along morphisms which are coverings for the fpqc topology} (but not necessarily faithfully flat morphisms)?

   Subsection \ref{reg} points out the specific issue of regularity,  which is settled by the following theorem.
   
      \begin{theorem}\label{T7}  Let $Y\stackrel{f}{\to} X$ be a morphism of affine Noetherian schemes which is a covering for the fpqc topology. 
      \begin{enumerate} 
      \item (Ma) If $\,Y$ is regular, so is $X$.
      \item  If in addition $f$ is a finite covering and $Y$ is connected, then $f$ is flat.  
      \end{enumerate}
   \end{theorem}
    \begin{proof} Let us write $R=\cO(X), S= \cO(Y)$. We may assume that $R$ is local with maximal ideal $\frak m$, and, replacing $S$ by its localization at some maximal ideal above $\frak m$, that $R\to S$ is a local homomorphism. Let $T$ be an $S$-algebra which is faithfully flat over $R$. 
    
    $(1)$\footnote{there are two well-known special cases where this holds: if $f$ is flat (by 
    \cite[IV. 6.5.2]{EGA}), or if $f$ admits a section (by \cite[IV .0.17.3.3]{EGA}). } This is a straightforward consequence of \cite[Cor 2.2]{BIM}, which claims the following: if $S/\frak m_S$ is of finite flat dimension $d$,  $U $ is an $S$-module such that $\frak m_S. U \neq U$, and $M$ is a finite $R$-module such that ${\rm{Tor}}^R_i(U, M)=0$ for $d+1$ consecutive values of $i$,  then $M$ has a finite free resolution. This applies to our case, with $d = \dim S$ since $S$ is regular, with $U=T$ which is faithfully flat over $R$, and with any $M$. We conclude that $R$ is regular. 
    
 $(2)$  By Theorem \ref{T5}, we know that $R$ is Cohen-Macaulay. The $\frak m$-adic completion $\hat T$ remains faithfully flat over $R$ (\cite[1.1.1]{A1}). In particular, $\hat T$ is a Cohen-Macaulay $R$-algebra, and since it is a $\frak m$-adically complete algebra over the finite  extension $S$ of $R$, it is also a Cohen-Macaulay $S$-algebra \cite{BaS}. Since $S$ is regular, $\hat T$ is faithfully flat over $S$. By Lemma~\ref{L1},
 it follows that  $R\to S$ is faithfully flat (and by \cite[IV. 6.5.2]{EGA} that $R$ is regular, which gives a second proof of point (1) under this additional assumption). 
    \end{proof}

 \begin{remark} In positive characteristic,  point $(1)$ is also a straightforward consequence of Theorem \ref{T6}: since $Y$ is regular, $Y\stackrel{F_Y}\to Y \stackrel{f}\to X$ is a composition of coverings for the fpqc topology, which factors through $F_X$, hence $F_X$ is a covering for the fpqc topology and by \ref{T6}, $X$ is regular.  
  \end{remark} 
  
   \begin{example}\label{exF}  In \cite[Cor 2.2]{BIM}, $U$ is an $S$-module, not necessarily an $S$-algebra. This allows to settle in the negative Ferrand's question (Remark~\ref{qF}), with the example of the double covering of the cone (cf. \ref{ex 4.7}): the local ring $k[[x,y]]$ is a pure extension of $k[[x^2, xy, y^2]] = k[[x,y]]^{{\mathbb  Z}/2}$, but there is no $k[[x,y]]$-module $U$ which is faithfully flat over $k[[x^2, xy, y^2]]$. 
  \end{example}

 \end{sloppypar}
\end{document}